\documentclass[12pt]{amsart}
\usepackage[margin=1in,letterpaper,portrait]{geometry}

\usepackage{latexsym,mathtools}   
\usepackage{epsfig,setspace}
\usepackage{a4}
\usepackage{amssymb}
\usepackage{amsthm}
\usepackage{amsbsy} 
\usepackage{upgreek}
\usepackage{relsize}
\usepackage{tikz-cd}
\usepackage{ytableau}
\usepackage{amsfonts,amssymb,latexsym,amscd,amsmath,euscript,enumerate,verbatim,calc}
\allowdisplaybreaks
\usepackage{url}
\usepackage{hhline}
\usepackage{pifont}
\usepackage{relsize}
\usepackage[toc,page]{appendix} 
\usepackage[colorlinks=true,linkcolor=blue,citecolor=magenta]{hyperref}

\theoremstyle{definition}

\newtheorem{theorem}{Theorem}[section]    
\newtheorem{lemma}[theorem]{Lemma}
\newtheorem{proposition}[theorem]{Proposition}

\newtheorem{problem}[theorem]{Problem}
 
\newtheorem{corollary}[theorem]{Corollary}
\newtheorem{defn}[theorem]{Definition}

\newtheorem{remark}[theorem]{Remark}
\newtheorem{example}[theorem]{Example}

\def\Fq{{\mathbb F}_q}

\newcommand{\gauss}[2]{\genfrac{[}{]}{0pt}{}{#1}{#2}_q}

\allowdisplaybreaks
\makeatletter
\def\imod#1{\allowbreak\mkern10mu({\operator@font mod}\,\,#1)} 
\makeatother

\title{Enumeration of splitting subsets of endofunctions on finite sets} 

\author{Divya Aggarwal} 
\address{Indraprastha Institute of Information Technology Delhi (IIIT-Delhi), New Delhi 110020, India.}
\email{divyaa@iiitd.ac.in} 



\begin{document}
\begin{abstract}
Let $d$ and $n$ be positive integers such that $d|n$. Let $[n]=\{1,2,\ldots,n\}$ and $T$ be an endofunction on $[n]$. A subset $W$ of $[n]$ of cardinality $n/d$ is said to be $d$-splitting if $W \cup TW \cup \cdots \cup T^{d-1}W =[n]$. Let $\sigma(d;T)$ denote the number of $d$-splitting subsets. If $\sigma(2;T)>0$, then we show that $\sigma(2;T)=g_T(-1)$, where $g_T(t)$ is the generating function for the number of $T$-invariant subsets of $[n]$. It is interesting to note that substituting a root of unity into a polynomial with integer coefficients has an enumerative meaning. 
More generally, let $g_T(t_1,\ldots,t_d)$ be the generating function for the number of $d$-flags of $T$-invariant subsets. We prove for certain endofunctions $T$, if 
$\sigma(d;T)>0$, then $\sigma(d;T)=g_T(\zeta,\zeta^2,\ldots,\zeta^d)$, where $\zeta$ is a primitive $d^{th}$ root of unity.
\end{abstract} 

\maketitle
\tableofcontents 

\section{Introduction}
Throughout the paper, let $d$, $k$, and $N$ denote positive integers. Let $[N]=\{1,\ldots,N\}$ and $n$ be a positive integer such that $d|n$. Let $\zeta$ be a primitive $d^{th}$ root of unity. Unless otherwise stated, $T$ will always denote an endofunction on $[N]$. 

We begin with the following definition.
\begin{defn}
Let $d$ and $n$ be positive integers such that $d|n$. Let $T$ be an endofunction on $[n]$. A subset $W$ of $[n]$ of cardinality $n/d$ is said to be $d$-splitting if
$$
W \cup TW \cup \cdots \cup T^{d-1}W=[n],
$$
where $T^i$ denotes the $i$-fold composition of $T$.
\end{defn}
We denote by $\sigma(d;T)$ the number of $d$-splitting subsets for the endofunction $T$.
\begin{problem}
\label{prob1}
For a given endofunction $T$, what is $\sigma(d;T)$?
\end{problem}


The $q$-analogue of this is a well-studied open problem, but it seems that Problem \ref{prob1} has not been studied particularly in the literature yet. The $q$-analogue of $[n]$ is the $n$-dimensional vector space over the finite field $\Fq$, while the $q$-analogue of an endofunction on $[n]$ is a linear operator on the vector space of dimension $n$ over $\Fq$ \cite[p. 89]{MR2868112}. The problem translates over the finite fields as follows: Let $V$ be a $dm$-dimensional vector space over the finite field $\Fq$ and let $T$ be a linear operator on $V$. An $m$-dimensional subspace $W$ of $V$ is said to be $T$-splitting if
$$
W+TW+\cdots+T^{d-1}W =V.
$$
Let $\sigma_q(d;T)$ denote the number of $m$-dimensional $T$-splitting subspaces. Then for an arbitrary assignment of the operator $T$, determination of $\sigma_q(d;T)$ is an open problem \cite{MR2961399}. An explicit formula for $\sigma_q(d;T)$ is known when $T$ has an irreducible characteristic polynomial \cite{MR4263652, MR3093853}, is regular nilpotent \cite{FFA1}, is regular diagonalizable \cite{prasad2021set, MR4490876}, or when the invariant factors of $T$ satisfy certain degree conditions \cite{FFA2}.

The case $d=2$ is of particular interest. A complete solution for $\sigma_q(2;T)$ for an arbitrary operator $T$ is recently given by Prasad and Ram \cite{prasad2023splitting}. 

Determining $\sigma(2;T)$ is essentially counting the number of subsets $W$ of $[2m]$, of cardinality $m$, such that $W$ is mapped to its complement under $T$. We answer this problem as follows. A subset $U$ of $[n]$ is said to be $T$-invariant if $TU \subseteq U$. Let $g_T(t)$ denote the generating function for the number of $T$-invariant subsets, i.e.
$$
g_T(t)=\sum_{i=0}^n a_i t^i,
$$
where $a_i$ is the number of $T$-invariant subsets of cardinality $i$. We prove that if $\sigma(2;T)>0$, then
\begin{align}
\label{d_is_2}
\sigma(2;T) = g_T(-1).
\end{align}

Note that $-1$ is the primitive second root of unity, and substituting the second root of unity into a polynomial with integer coefficients counts the enumerative measure $\sigma(2;T)$. 
Cyclic Sieving Phenomenon (CSP) is a similar phenomenon studied by Riener, Stanton and White \cite{MR3156682}. Cyclic sieving is a phenomenon by which evaluating a generating function for a finite set at the roots of unity counts symmetry classes of objects acted on by a cyclic group. It generalizes Stembridge's $q=-1$ phenomenon \cite{MR1262215, MR1297179, MR1387685}. Let $C$ be a cyclic group generated by an element $c$ of order $n$. Suppose $C$ acts on a set $X$. Let $X(q)$ be a polynomial with integer coefficients. Then the triple $(X,X(q),C)$ is said to exhibit the cyclic sieving phenomenon if, for all integers $d$, the value $X(e^{2\pi id/n})$ is the number of elements fixed by $c^d$. 
We refer to the survey article of Sagan \cite{MR2866734} for more on this topic.



For a general $d$, we show in section \ref{sec 3} that when the endofunction $T$ is a cycle or a chain, then $\sigma(d;T)=g_T(\zeta, \zeta^2,\ldots,\zeta^d)$, where $\zeta$ is a primitive $d^{th}$ root of unity and $g_T(t_1,\ldots,t_d)$ is the generating function for the number of $d$-flags of $T$-invariant subsets.
A $d$-flag of $T$-invariant subsets is an increasing sequence of subsets 
$\emptyset=U_0 \subseteq U_1 \subseteq \cdots \subseteq U_{d-1} \subseteq U_d=[n]$, where each $U_i$ is $T$-invariant (i.e. $TU_i \subseteq U_i$). 
We write the generating function for the number of $d$-flags of $T$-invariant subsets as follows:
$$
g_T(t)=g_T(t_1,\ldots,t_{d})=\sum_{J} a_J t^J,
$$ 
where the summation index $J$ runs over all d-tuples of non-negative integers and $t^J$ means ${t_1}^{j_1} {t_2}^{j_2} \ldots {t_d}^{j_d}$. Here $a_J$ is the number of flags of $T$-invariant subsets $\emptyset=U_0 \subseteq U_1 \subseteq \cdots \subseteq U_{d-1}\subseteq U_d=[n]$ such that 
$$
|U_i|=j_1+\cdots+j_i ~~ \forall ~ 1 \leq i \leq d.
$$
We prove that if $T$ is a tree (see section \ref{sec 2} for the definition) and $\sigma(d;T)>0$, then 
$
\sigma(d;T)=g_T(\zeta, \zeta^2,\ldots,\zeta^d)
$.

In Section \ref{sec 4}, we extend our result to endofunctions that satisfy certain structure criteria.
More precisely, let $T$ be an endofunction such that:\\
{\it I.} $T$ has a central cycle consisting of $ds$ nodes $(s \geq 0)$ and $k$ trees attached to the nodes of the cycle such that each attached tree has a $d$-splitting subset.\\
{\it II.} $T$ has a central cycle consisting of $ds+1$ nodes and $T_1, \ldots, T_k$ are $k$ trees attached to the nodes of the cycle such that $T_1$, together with its root node on the cycle, has a $d$-splitting subset and each of $T_2, \ldots, T_k$ has a $d$-splitting subset.\\
If $T$ is either of Type {\it I} or of Type {\it II}, then $$\sigma(d;T)=g_T(\zeta, \zeta^2,\ldots,\zeta^d).$$

\section{The case $d=2$}
\label{sec 2}
We begin by defining the structures chains, cycles, and trees. A chain on $[N]$ is defined as $C: a_1 \rightarrow a_2 \rightarrow \cdots \rightarrow a_{N-1} \rightarrow a_N$, where each $a_i \in [N]$ and $a_i \neq a_j$ when $i \neq j$. A cycle on $[N]$ is a permutation $P$ consisting of precisely one cycle when $P$ is written as a product of disjoint cycles. An example of a cycle is shown in Fig. \ref{fig_cycle}. A tree is an acyclic-connected simple graph. A directed tree is a directed acyclic graph whose underlying graph is a tree. For our purposes, we will only be considering directed trees, and with a slight abuse of notation, we will call directed trees as trees.

\begin{defn}
A rooted tree $T$ is a tree with a distinguished node, called the root node, $R$, such that all the edges point towards the root (see Fig. \ref{rooted tree}). The trees $T_1, \ldots, T_k$ in Fig. \ref{rooted tree} are called subtrees of $T$.
\begin{figure}[h]
 \caption{A rooted tree T.}
      \label{rooted tree}
  \begin{center}
    \begin{minipage}{\textwidth}
      \centering
      \begin{tikzpicture}
[scale=1,every node/.style={circle,fill=black},inner sep=2.5pt, minimum size=6pt]
        \node [label=below:$R$] (1) at (0,0) {};
        \node[label=right:$T_2$] (2) at (-1,1) {};
        \node[label=right:$T_{k-1}$] (3) at (1,1) {};
        \node[label=right:$T_1$] (4) at (-2,1) {};
        \node[label=right:$T_k$] (5) at (2.5,1) {};
\node[style={circle,fill=white}] (6) at (-0.1,1) {};
\node[style={circle,fill=white}] (7) at (0.7,1) {};
        \draw[-{Stealth[slant=0]}]
        (2) to  (1) ;
        \draw[-{Stealth[slant=0]}]
        (3) to  (1);
  \draw[-{Stealth[slant=0]}]
        (4) to  (1) ;
  \draw[-{Stealth[slant=0]}]
        (5) to  (1) ;
\draw[dashed]
        (6) to  (7) ;
      \end{tikzpicture}
\end{minipage}
  \end{center}
\end{figure}
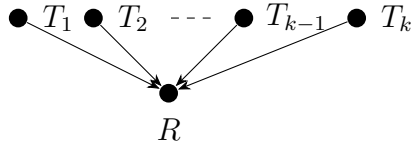
\end{defn}
We adopt the following convention.\\
Let $C$ be a chain with $N$ nodes. If $C$ does not feed into a cycle, we assume that the last node of $C$ goes to itself, thereby making $C$ an endofunction on $[N]$. Figure \ref{fig_chain} is an example of a chain on $[4]$.
Likewise, let $T$ be a rooted tree with the root node $R$. If $R$ does not feed into a cycle, we assume that $R$ is mapped to itself so that $T$ becomes an endofunction.

Let $T$ be an endofunction on $[N]$. Let us recall the $T$-invariant subsets of $[N]$.
\begin{defn}
A subset $U$ of $[N]$ is said to be $T$-invariant if $TU \subseteq U$, where $TU$ denotes the image of $U$ under $T$. We denote by $g_T(t)$, the generating function for the number of $T$-invariant subsets, i.e.
$$
g_T(t)=\sum_{i=0}^N a_i t^i,
$$
where $a_i$ is the number of $T$-invariant subsets of cardinality $i$.
\end{defn}
Note that $g_T(t)$ is a polynomial with integer coefficients. 
The following theorem is the main result of this section.
\begin{theorem}
\label{main}
Let $T$ be an endofunction on $[2m]$ for which $\sigma(2;T)>0$. Then
$$
\sigma(2;T)=g_T(-1).
$$
\end{theorem}

We use the following result from Bergeron, Labelle, and Leroux \cite[p. 41]{MR1629341} to prove Theorem \ref{main}.
Every endofunction is a permutation of disjoint rooted trees. Figure \ref{endofunction} shows that an endofunction $T$ can naturally be identified with a permutation of disjoint rooted trees, where each rooted tree is shown in a different colour. The nodes on the cycle serve as the roots of the attached trees.

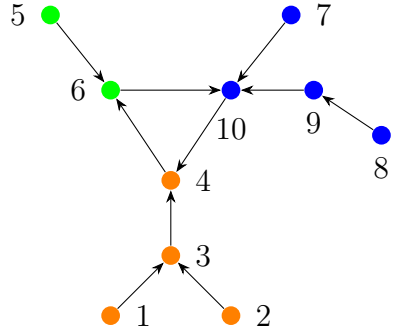
\begin{figure}[h]
 \caption{An endofunction as a permutation of disjoint rooted trees.}
      \label{endofunction}
  \begin{center}
    \begin{minipage}{\textwidth}
      \centering
      \begin{tikzpicture}
[scale=1,every node/.style={circle,fill=black,,inner sep=2.5pt, minimum size=6pt}]
        \node[label=right:$4$,style={circle,fill=orange}] (4) at (0,-0.2) {};
      \node[label=right:$3$,style={circle,fill=orange}] (3) at (0,-1.2) {};
  \node[label=right:$1$,style={circle,fill=orange}] (1) at (-0.8,-2) {};
  \node[label=right:$2$,style={circle,fill=orange}] (2) at (0.8,-2) {};
        \node[label=left:$6$,style={circle,fill=green}] (6) at (-0.8,1) {};
        \node[label=below:$10$,style={circle,fill=blue}] (10) at (0.8,1) {};
  \node[label=right:$7$,style={circle,fill=blue}] (7) at (1.6,2) {};
  \node[label=left:$5$,style={circle,fill=green}] (5) at (-1.6,2) {};
     \node[label=below:$9$,style={circle,fill=blue}] (9) at (1.9,1) {};
\node[label=below:$8$,style={circle,fill=blue}] (8) at (2.8,0.4) {};
        \draw[-{Stealth[slant=0]}]
        (6) to  (10) ;
        \draw[-{Stealth[slant=0]}]
        (10) to  (4);
  \draw[-{Stealth[slant=0]}]
        (3) to  (4);
   \draw[-{Stealth[slant=0]}]
        (4) to  (6);
\draw[-{Stealth[slant=0]}]
        (8) to  (9);
\draw[-{Stealth[slant=0]}]
        (9) to  (10);
\draw[-{Stealth[slant=0]}]
        (7) to  (10);
\draw[-{Stealth[slant=0]}]
        (1) to  (3);
\draw[-{Stealth[slant=0]}]
        (2) to  (3);
\draw[-{Stealth[slant=0]}]
        (5) to  (6);
      \end{tikzpicture}
\end{minipage}
  \end{center}
\end{figure}

The following results depict the generating functions for the number of $T$-invariant subsets when $T$ is a cycle or a chain.
\begin{proposition}
Let $T$ be a cycle on $[N]$. Then
$$
g_T(t)=1+t^N.
$$
\end{proposition}

\begin{proof}
Since $T$ is a cycle, there are only two $T$-invariant subsets: the empty set and the whole set $[N]$.
\end{proof}

\begin{proposition}
\label{gf chain}
Let $T$ be a chain on $[N]$. Then
$$
g_T(t)=1+t+t^2+\cdots+t^N.
$$
\end{proposition}

\begin{proof}
For each $i~ (0 \leq i \leq N)$, there exists precisely one $T$-invariant subset of cardinality $i$, consisting of the last $i$ nodes of the chain. By the structure of $T$, it is easy to see that no other subset of $[N]$ is $T$-invariant.
\end{proof}

\begin{example}
Let $T$ be given by the following chain (see Figure \ref{fig_chain}). $$T:1\rightarrow 2 \rightarrow 3 \rightarrow 4 \rightarrow 4.$$ 
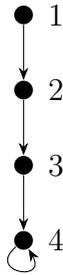
\begin{figure}[h]
\caption{A chain on $[4]$.}
\label{fig_chain}
  \begin{center}
    \begin{minipage}{\textwidth}
      \centering
      \begin{tikzpicture}
        [scale=1,every node/.style={circle,fill=black},inner sep=2.5pt, minimum size=6pt]
        \node[label=right:$4$] (4) at (1,0) {};
        \node[label=right:$3$] (3) at (1,1) {};
        \node[label=right:$2$] (2) at (1,2) {};
        \node[label=right:$1$] (1) at (1,3) {};
        \draw[-{Stealth[slant=0]}]
        (1) to  (2) ;
        \draw[-{Stealth[slant=0]}]
        (2) to  (3);
  \draw[-{Stealth[slant=0]}]
        (3) to  (4);
\draw [-{Stealth[slant=0]}]
(4) to [out=220,in=300,looseness=8] (4);
      \end{tikzpicture}
\end{minipage}
  \end{center}
\end{figure}
Then there exists a unique $T$-invariant subset of cardinality $i ~(0 \leq i \leq 4)$, namely the last $i$ nodes of the chain. So, $g_T(t)=1+t+t^2+t^3+t^4$.
\end{example}
Proposition \ref{gf chain} may be used to recursively obtain the generating function for a tree. 
\begin{lemma}
\label{gf_tree}
Let $T$ be a rooted tree on $[N]$ with root node $R$ and $k$ subtrees $T_1, \ldots, T_k$. Then 
$$
g_T(t)=1+t\prod_{i=1}^k g_{T_i}(t),
$$
where $g_{T_i}(t)$ denotes the generating function for the number of $T_i$ -invariant subsets.
\end{lemma}
\begin{proof}
As the empty set is always $T$-invariant, we get $t^0$ in the generating function. Since all the nodes of the tree eventually feed into the root node $R$, therefore $R$ must belong to every non-empty $T$-invariant subset. So, we get $t$ times the product of the generating functions for each subtree, as $T_i$ is independent of $T_j$, for $i \neq j$.
\end{proof}
Combining the above results, we obtain the following generating function for the number of $T$-invariant subsets for a general endofunction $T$.

\begin{lemma}
\label{genfn}
Let $T$ be an endofunction on $[N]$ with $s$ connected components. Let the central cycle of $i^{th}$ component of $T$ has $r_i$ nodes, and $T_{i,1},\ldots,T_{i,k_i}$ be $k_i$ trees attached to the cycle of $i^{th}$ component. Then
$$
g_T(t)=\prod_{i=1}^{s} \left(1+t^{r_i}\prod_{j=1}^{k_i}g_{T_{i,j}}(t)\right),
$$
where $g_{T_{i,j}}(t)$ denotes the generating function for the number of $T_{i,j}$-invariant subsets.
\end{lemma}

\begin{proof}
Since each connected component is independent of the other, the result immediately follows by Lemma \ref{gf_tree}, as all the nodes of the cycle must belong to every non-empty $T$-invariant subset.
\end{proof}
We adopt the following notation for trees.

\begin{defn}
A node $\tau$ of a tree $T$ is said to be a {\it branching node} if there exist nodes $\tau_1$ and $\tau_2$ such that both $\tau_1$ and $\tau_2$ feed into $\tau$ under the action of $T$.
\end{defn}
Figure \ref{branching node} represents the branching node $\tau$.

\begin{figure}[h]
 \caption{Branching node $\tau$}
      \label{branching node}
  \begin{center}
    \begin{minipage}{\textwidth}
      \centering
      \begin{tikzpicture}
[scale=1,every node/.style={circle,fill=black},inner sep=2.5pt, minimum size=6pt]
        \node[label=right:$\tau$] (3) at (0,0) {};
        \node[label=right:$\tau_1$] (1) at (-1,1) {};
        \node[label=right:$\tau_2$] (2) at (1,1) {};
        \draw[-{Stealth[slant=0]}]
        (1) to  (3) ;
        \draw[-{Stealth[slant=0]}]
        (2) to  (3);
      \end{tikzpicture}
\end{minipage}
  \end{center}
\end{figure}
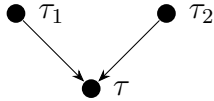

\begin{defn}
Let $T$ be a tree. To obtain the {\it chains of $T$}, apply the following method. Begin with the leaves of $T$ and go up to the first branching node (where the branching node is excluded). Cut off all the chains so obtained to get a forest. Repeat the above procedure with each tree in the forest. In a finite number of steps, we have only the structures of the kind $\tau_1 \rightarrow \tau_2 \rightarrow \cdots \rightarrow \tau_k$. Each such structure is called a chain of the tree $T$.
\end{defn}
Figure \ref{fig_chains} depicts the various chains of a tree, where each chain has been highlighted in a different colour.
\begin{figure}[h]
 \caption{Chains of a tree.\\ $C_1:3; C_2:1\rightarrow 2;  C_3: 5 \rightarrow 6 \rightarrow 7; C_4: 4 \rightarrow 8 \rightarrow 9$}
      \label{fig_chains}
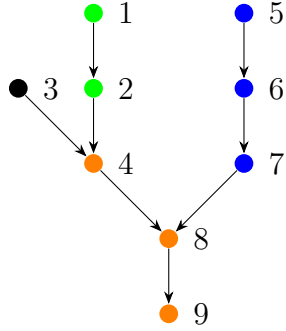
  \begin{center}
    \begin{minipage}{\textwidth}
      \centering
      \begin{tikzpicture}
[scale=1,every node/.style={circle,fill=black},inner sep=2.5pt, minimum size=6pt]
        \node[label=right:$8$,style={circle,fill=orange}] (8) at (0,0) {};
        \node[label=right:$9$,style={circle,fill=orange}] (9) at (0,-1) {};
        \node[label=right:$4$,style={circle,fill=orange}] (4) at (-1,1) {};
        \node[label=right:$7$,style={circle,fill=blue}] (7) at (1,1) {};
\node[label=right:$6$,style={circle,fill=blue}] (6) at (1,2) {};
\node[label=right:$5$,style={circle,fill=blue}] (5) at (1,3) {};
      \node[label=right:$1$,style={circle,fill=green}] (1) at (-1,3) {};
 \node[label=right:$2$,style={circle,fill=green}] (2) at (-1,2) {};
\node[label=right:$3$] (3) at (-2,2) {};
        \draw[-{Stealth[slant=0]}]
        (4) to  (8) ;
        \draw[-{Stealth[slant=0]}]
        (7) to  (8);
  \draw[-{Stealth[slant=0]}]
        (8) to  (9) ;
 \draw[-{Stealth[slant=0]}]
        (5) to  (6) ;
 \draw[-{Stealth[slant=0]}]
        (6) to  (7) ;
 \draw[-{Stealth[slant=0]}]
        (1) to  (2) ;
 \draw[-{Stealth[slant=0]}]
        (2) to  (4) ;
\draw[-{Stealth[slant=0]}]
        (3) to  (4) ;
      \end{tikzpicture}
\end{minipage}
  \end{center}
\end{figure}
To prove Theorem \ref{main}, we require the following lemmas.
\begin{lemma}
Let $T$ be a rooted tree with $k$ subtrees $T_1,\ldots,T_k$.
\label{star}
A $d$-splitting subset for $T$ exists if and only if there exists a unique subtree $T_i$ such that $\widetilde{T_i}( = T_i$ with the root node $R$ (Fig. \ref{tilde})) has a $d$-splitting subset, and for $i \neq j$, $T_j$ has a $d$-splitting subset.
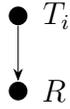
\begin{figure}[h]
 \caption{$\widetilde{T_i} =T_i \rightarrow R$}
      \label{tilde}
  \begin{center}
    \begin{minipage}{\textwidth}
      \centering
      \begin{tikzpicture}
[scale=1,every node/.style={circle,fill=black},inner sep=2.5pt, minimum size=6pt]
        \node [label=right:$R$] (1) at (0,0) {};
        \node[label=right:$T_i$] (2) at (0,1) {};
      
        \draw[-{Stealth[slant=0]}]
        (2) to  (1) ;
 
      \end{tikzpicture}
\end{minipage}
  \end{center}
\end{figure}
\end{lemma}
\begin{lemma}
\label{induc}
For every tree $T$ on $[N]$ which has a $2$-splitting subset, the following holds $$g_T(-1)=1.$$
\end{lemma}
\begin{proof}
We prove by induction on the number of chains of the tree $T$.
If there is one chain, then
$$
g_T(t)=1+t+\cdots+t^N.
$$
Since $T$ has a $2$-splitting subset, $N$ is even. Therefore $g_T(-1)=1$. Suppose the result holds for every tree having less than $k$ chains.
Let $T$ be a tree for which a $2$-splitting subset exists and has $k$ chains. Then $T=\widetilde{T} \leftarrow C$,
where $C$ is a chain starting from a leaf and ending to a branching node (branching node excluded), having an even number of nodes, and $\widetilde{T}$ is a tree. Such a choice of $C$ is always possible since $T$ has a $2$-splitting subset and whenever a branching occurs in $T$, then  one child is odd and others are even (see Lemma \ref{star}). Existence of a $2$-splitting subset for $T$ implies that $\widetilde{T}$ also has a $2$-splitting subset, and $g_{\widetilde{T}}(-1)=1$ by induction hypothesis. Since $g_C(-1)=1$, the result follows.
\end{proof}

The following corollary will be useful to prove the main result.
\begin{corollary}
\label{add_node}
Let $T$ be a tree such that $T$ doesn't have a $2$-splitting subset and let $\widetilde{T}$ be a rooted tree obtained by joining $T$ to a root node $R$ such that $\widetilde{T}$ has a splitting subset, then $f_T(-1)=0$.
\end{corollary}

\begin{proof}
By Lemma \ref{induc}, $f_{\widetilde{T}}(-1)=1$ and $f_{\widetilde{T}}(t)=1+t f_{T}(t)$. Therefore $f_T(-1)=0$.
\end{proof}
The next proposition shows that Theorem \ref{main} holds for trees.
\begin{proposition}
\label{unique tree}
If a $2$-splitting subset exists for a tree, then it is unique.
\end{proposition}
\begin{proof}
Let $T$ be a tree for which a $2$-splitting subset exists. Leaves of the tree must belong to the splitting subset since they are not fed by any other node. Cut off the leaves along with the nodes to which they are mapped, to obtain a forest. Now each tree in the forest has a $2$-splitting subset. Repeat the above process with each tree of the forest until we are left with chains having two nodes, each of which has exactly one $2$-splitting subset, namely the upper node of the chain. This shows that the $2$-splitting subset for $T$ is unique.
\end{proof}

We are finally ready to prove Theorem \ref{main}.

\begin{proof}[Proof of Theorem \ref{main}]
Let $T$ be an endofunction such that $\sigma(2;T)>0$. It is enough to prove the result when $T$ is connected. Let $T$ has a central cycle with $k$ trees $T_1,\ldots,T_k$ attached to the nodes of the cycle. The following two cases arise.\\
{\it Case 1: Each tree $T_1,\ldots,T_k$ has a $2$-splitting subset.}\\
Since $\sigma(2;T)>0$ and each attached tree has a $2$-splitting subset, it follows by Proposition \ref{unique tree} that the number of nodes on the cycle is even. In this case, $\sigma(2;T)=2$ since we have two sets of choices to select the nodes for $2$-splitting subsets (see, for example, Figures \ref{w1} and \ref{w2}). By Lemma \ref{genfn},
$$
g_T(-1)=1+(-1)^{r}\prod_{i=1}^{k}g_{T_i}(-1),
$$
where $r$ is the number of nodes on the cycle.
Since $g_{T_i}(-1)=1$ for each $i$, the result follows.\\
{\it Case 2: Some of the trees among $T_1,\ldots,T_k$ have a $2$-splitting subset (only when) combined with the node of the cycle to which they are attached.}\\
Since $\sigma(2;T)>0$, the trees which do not have a $2$-splitting subset must have a $2$-splitting subset together with the node of the cycle to which they are attached. Moreover, at each cycle node, there could be at most one such tree (Lemma \ref{star}). If the cycle has an even (odd) number of nodes, then the number of such trees is even (odd). Also, such trees occur at even gaps (i.e. there exists an even number of nodes in the cycle between the nodes to which such trees are attached). This gives a unique choice for the $2$-splitting subset, thereby $\sigma(2;T)=1$. And
$$
g_T(-1)=1+(-1)^{r}\prod_{i=1}^{k}f_{T_i}(-1)=1,
$$
where the last equality follows by Corollary \ref{add_node} since $f_{T_j}(-1)=0$ for some $1\leq j \leq k$.
\end{proof}

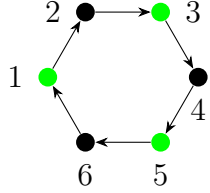
\begin{figure}
\caption{$W_1=\{1,3,5\}$.}
\label{w1}
 \begin{minipage}{\textwidth}
      \centering
      \begin{tikzpicture}
        [scale=1,every node/.style={circle,fill=black},inner sep=2.5pt, minimum size=6pt]
\node[style={circle,fill=green},label=left:$1$] (1) at (-1,0) {};
        \node[label=left:$2$] (2) at (-0.5,0.866) {};
        \node[style={circle,fill=green},label=right:$3$] (3) at (0.5,0.866) {};
        \node[label=below:$4$] (4) at (1,0) {};
        \node[style={circle,fill=green},label=below:$5$] (5) at (0.5,-0.866) {};
        \node[label=below:$6$] (6) at (-0.5,-0.866) {};
\draw[-{Stealth[slant=0]}]
        (1)  to  (2);
        \draw[-{Stealth[slant=0]}]
        (2)  to  (3);
        \draw[-{Stealth[slant=0]}]
        (4) to  (5);
 \draw[-{Stealth[slant=0]}]
        (5) to (6);
 \draw[-{Stealth[slant=0]}]
        (3)  to  (4);
 \draw[-{Stealth[slant=0]}]
        (6)  to  (1);
      \end{tikzpicture}
\end{minipage}
\end{figure}

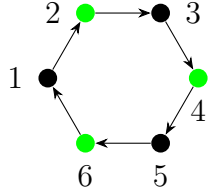
\begin{figure}[h]
\caption{$W_2=\{2,4,6\}$.}
\label{w2}
  \begin{minipage}{\textwidth}
      \centering
      \begin{tikzpicture}
        [scale=1,every node/.style={circle,fill=black},inner sep=2.5pt, minimum size=6pt]
\node[label=left:$1$] (1) at (-1,0) {};
        \node[style={circle,fill=green},label=left:$2$] (2) at (-0.5,0.866) {};
        \node[label=right:$3$] (3) at (0.5,0.866) {};
        \node[style={circle,fill=green},label=below:$4$] (4) at (1,0) {};
        \node[label=below:$5$] (5) at (0.5,-0.866) {};
        \node[style={circle,fill=green},label=below:$6$] (6) at (-0.5,-0.866) {};
\draw[-{Stealth[slant=0]}]
        (1)  to  (2);
        \draw[-{Stealth[slant=0]}]
        (2)  to  (3);
        \draw[-{Stealth[slant=0]}]
        (4) to  (5);
 \draw[-{Stealth[slant=0]}]
        (5) to (6);
 \draw[-{Stealth[slant=0]}]
        (3)  to  (4);
 \draw[-{Stealth[slant=0]}]
        (6)  to  (1);
      \end{tikzpicture}
\end{minipage}
\end{figure}

\begin{remark}
We remark that if $\sigma(2;T)=0$ for an endofunction $T$, then $g_{T}(-1)$ may be zero or non-zero. We provide examples for both cases.\\ \\
{\it Case 1: $\sigma(2;T)=0$ and $g_T(-1)=0$.}\\
Let $T:[4] \rightarrow [4]$ be defined as $1 \rightarrow 3$, $2 \rightarrow 3$, $3 \rightarrow 4$, and $4 \rightarrow 4$ (see Fig. \ref{fig 1}). Here $g_T(t)=1+t+t^2(1+t)^2$. Clearly $\sigma(2,2;T)=0$ by Lemma \ref{star} and $g_T(-1)=0$ as well.
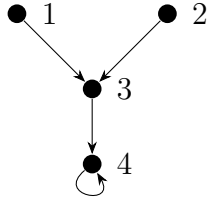
\begin{figure}[h]
 \caption{$\sigma(2;T)=0$ and $g_T(-1)=0$}
      \label{fig 1}
  \begin{center}
    \begin{minipage}{\textwidth}
      \centering
      \begin{tikzpicture}
[scale=1,every node/.style={circle,fill=black},inner sep=2.5pt, minimum size=6pt]
        \node[label=right:$3$] (3) at (0,0) {};
        \node[label=right:$4$] (4) at (0,-1) {};
        \node[label=right:$1$] (1) at (-1,1) {};
        \node[label=right:$2$] (2) at (1,1) {};
     
        \draw[-{Stealth[slant=0]}]
        (1) to  (3) ;
        \draw[-{Stealth[slant=0]}]
        (2) to  (3);
  \draw[-{Stealth[slant=0]}]
        (3) to  (4) ;
\draw [-{Stealth[slant=0]}]
(4) to [out=220,in=300,looseness=8] (4);
      \end{tikzpicture}
\end{minipage}
  \end{center}
\end{figure}\\
{\it Case 2: $\sigma(2;T)=0$ and $g_T(-1) \neq 0$.}\\
Consider the endofunction $T:[4] \rightarrow [4]$ given by $1 \rightarrow 4$, $2 \rightarrow 4$, $3 \rightarrow 4$, and $4 \rightarrow 4$ (see Fig. \ref{fig 2}). Again by Lemma \ref{star}, $\sigma(2,2;T)=0$. But $g_T(t)=1+t(1+t)^3$, so $g_T(-1)=1 \neq 0$.
\begin{figure}[h]
 \caption{$\sigma(2;T)=0$ and $g_T(-1) \neq 0$}
     \label{fig 2}
  \begin{center}
    \begin{minipage}{\textwidth}
      \centering
      \begin{tikzpicture}
[scale=1,every node/.style={circle,fill=black},inner sep=2.5pt, minimum size=6pt]
        \node[label=right:$4$] (4) at (0,0) {};
        \node[label=right:$1$] (1) at (-1.5,1) {};
        \node[label=right:$2$] (2) at (0,1) {};
        \node[label=right:$3$] (3) at (1.5,1) {};
     
        \draw[-{Stealth[slant=0]}]
        (1) to  (4) ;
        \draw[-{Stealth[slant=0]}]
        (2) to  (4);
  \draw[-{Stealth[slant=0]}]
        (3) to  (4) ;
\draw [-{Stealth[slant=0]}]
(4) to [out=220,in=300,looseness=8] (4);
      \end{tikzpicture}
\end{minipage}
  \end{center}
\end{figure}
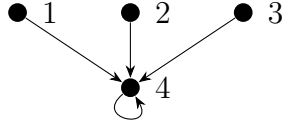
\end{remark}

\section{Cycles and trees}
\label{sec 3}
In this section, we consider the general case of $d$-splitting subsets. The result for $d=2$ case does not hold in general when $d>2$. We shall prove that if $T$ is either a cycle or a chain, then $\sigma(d;T)=g_T(\zeta, \zeta^2, \ldots, \zeta^d)$, where $g_T(t_1,\ldots,t_d)$ is the number of $d$-flags of $T$-invariant subsets (see Definition \ref{def}). We will further prove that if $T$ is a tree such that $\sigma(d;T)>0$, then $\sigma(d;T)=g_T(\zeta, \zeta^2, \ldots, \zeta^d)$. Recall from the introduction the definition of flags of $T$-invariant subsets.

\begin{defn}
Let $T$ be an endofunction on $[N]$. An increasing sequence of subsets of $[N]$,
$$
\emptyset = U_0 \subseteq U_1 \subseteq \cdots \subseteq U_{d-1} \subseteq U_d =[N]
$$
is said to be a $d$-flag \cite[p. 100]{MR2868112} of $T$-invariant subsets if $TU_i  \subseteq U_i$ for all $1 \leq i \leq d$.
\end{defn}
We will denote by a {\it flag}, a {\it $d$-flag of $T$-invariant subsets}, when the length of the flag is clear from the context. Next, we define the generating function for the number of flags.

\begin{defn}
\label{def}
Let $T$ be an endofunction on $[N]$. The generating function for the number of $d$-flags of $T$-invariant subsets is defined as:
$$
g_T(t)=g_T(t_1,\ldots,t_d) = \sum_{J} a_{J} t^J,
$$
where the summation runs over all $d$-tuples of non-negative integers $J=(j_1,\ldots,j_d)$ and $t^J$ denotes $t_1^{j_1}  \cdots t_d^{j_d}$. Here $a_J$ is the number of flags $\emptyset = U_0 \subseteq U_1 \subseteq \cdots \subseteq U_{d-1} \subseteq U_d=[N]$ such that $|U_i|=j_1+\cdots+j_i ~\forall ~1 \leq i \leq d$.
\end{defn}

\begin{example}
Let $T:[4] \rightarrow [4]$ be the cycle $1 \rightarrow 2 \rightarrow 3 \rightarrow 4 \rightarrow 1$ as shown in Figure \ref{fig_cycle}.
\begin{figure}[h]
\caption{A cycle with 4 nodes.}
\label{fig_cycle}
  \begin{center}
    \begin{minipage}{\textwidth}
      \centering
      \begin{tikzpicture}
        [scale=1,every node/.style={circle,fill=black},inner sep=2.5pt, minimum size=6pt]
        \node[label=below:$1$] (1) at (-1,0) {};
        \node[label=below:$3$] (3) at (1,0) {};
        \node[label=below:$2$] (2) at (0,1) {};
        \node[label=below:$4$] (4) at (0,-1) {};
        \draw[-{Stealth[slant=0]}]
        (1)  to  (2);
        \draw[-{Stealth[slant=0]}]
        (4)to  (1);
 \draw[-{Stealth[slant=0]}]
        (2)to  (3);
 \draw[-{Stealth[slant=0]}]
        (3)  to  (4);
      \end{tikzpicture}
\end{minipage}
  \end{center}
\end{figure}
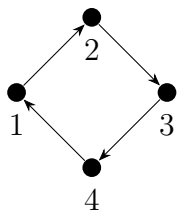
Let $\emptyset = U_0 \subseteq U_1 \subseteq \cdots \subseteq U_d=[4]$ be a flag of $T$-invariant subsets. If $|U_i| \geq 1$ for some $i$, then it is evident that $|U_i|=4$. Therefore for each $1 \leq i \leq d$, we have one flag such that $|U_i|=4$, $|U_j|=0$ for $j<i$. Hence $g_T(t)=g_T(t_1,\ldots,t_d)=t_1^4 +\cdots + t_d^4$.
\end{example}

The following result shows that this holds in general for cycles.

\begin{proposition}
\label{gf_cycle}
Let $T$ be a cycle with $N$ nodes. Then the generating function for the number of $d$-flags of $T$-invariant subsets is
$$
g_T(t)=g_T(t_1,\ldots,t_d)=t_1^N +\cdots + t_d^N.
$$
\end{proposition}

\begin{proof}
Let $\emptyset = U_0 \subseteq U_1 \subseteq \cdots \subseteq U_{d-1} \subseteq U_d=[N]$ be a flag of $T$-invariant subsets. If $|U_i| \geq 1$ for some $i$, then since $TU_i \subseteq U_i$, it follows that $|U_i|=N$. Therefore for each $1 \leq i \leq d$, we have one flag such that $|U_i|=N$, $|U_j|=0$ for $j<i$.
\end{proof}

Recall the complete symmetric polynomials \cite[p. 383]{MR2777360}. For fixed $k \geq 1$, the polynomial
$$
h_k(x_1,\ldots,x_N) = \sum_{1 \leq i_1 \leq i_2 \leq \cdots \leq i_k \leq N} x_{i_1} x_{i_2} \cdots x_{i_k}
$$
is called the complete homogeneous symmetric polynomial in $N$ variables. The generating function for $h_k(x_1,\ldots,x_N)$ \cite[p. 63]{MR4057234} is 
\begin{align}
\label{complete}
\sum_{k=0}^{\infty}h_k(x_1,\ldots,x_N) t^k=\prod_{i=1}^N \frac{1}{1-x_i t}.
\end{align}
The next proposition illustrates the generating function for the number of $d$-flags of $T$-invariant subsets when $T$ is a  chain.

\begin{proposition}
\label{gf_chain}
Let $T$ be a chain of length $k$. Then 
$$
g_T(t)=g_T(t_1,\ldots,t_d)= h_k(t_1,\ldots,t_d),
$$
where $h_k(t_1,\ldots,t_d)$ is the complete homogeneous symmetric polynomial of degree $k$ in the variables $t_1, \ldots, t_d$.
\end{proposition}

\begin{proof}
Let $\emptyset = U_0 \subseteq U_1 \subseteq \cdots \subseteq U_{d-1} \subseteq U_d=[k]$ be a flag of $T$-invariant subsets with 
$$
|U_i|=j_1+\cdots+j_i ~ \forall ~ i.
$$
Since $k= |U_d|=j_1+\cdots+j_d$, it follows that each term in the generating function $g_T(t)$ must be of degree $k$. As $T$ is a chain, for any $1 \leq s \leq k$, there is a unique $T$-invariant subset of cardinality $s$, namely, the last $s$ nodes of the chain. Therefore for any $d$-tuple of non-negative integers $(j_1,\ldots,j_d)$, there is a unique flag of $T$-invariant subsets $\emptyset = U_0 \subseteq U_1 \subseteq \cdots \subseteq U_{d-1} \subseteq U_d=[k]$ such that $|U_i|=j_1+\cdots+j_i$, and hence the coefficient of $t_1^{j_1} \cdots t_d^{j_d}$ is $1$.
\end{proof}
The following corollary is an immediate consequence of the above result, which we will use repeatedly.
\begin{corollary}
\label{zeta}
Let $T$ be a chain on $[n]$. Then
$$
g_T(\zeta,\zeta^2,\ldots,\zeta^d)=1.
$$
\end{corollary}

\begin{proof}
By Proposition \ref{gf_chain}, $g_T(\zeta,\zeta^2,\ldots,\zeta^d)=h_n(\zeta,\zeta^2,\ldots,\zeta^d)$. By \eqref{complete},
\begin{align*}
\sum_{k=0}^{\infty}h_k(\zeta,\zeta^2,\ldots,\zeta^{d}) t^k=&\prod_{i=1}^d \frac{1}{1-\zeta^{i} t}\\
=&\frac{1}{1-t^d} = \sum_{k=0}^{\infty} t^{dk}.
\end{align*}
Therefore $h_n(\zeta,\zeta^2,\ldots,\zeta^{d})=1$ as $n=dm$.
\end{proof}

The following proposition describes the generating function for the number of $d$-flags of $T$-invariant subsets for a rooted tree $T$.

\begin{proposition}
\label{gf tree}
Let $T$ be a rooted tree (Fig. \ref{rooted tree}) with $N$ nodes and $k$ subtrees $T_1,\ldots, T_k$. Then
\begin{align}
\label{eq gf tree}
g_T(t_1,\ldots,t_d) =\sum_{i=1}^d t_i \prod_{j=1}^k  g_{T_j}(t_i,t_{i+1},\ldots,t_d),
\end{align}
where $g_{T_j}(t_i, t_{i+1},\ldots,t_d)$ is the generating function for the number of $(d-i+1)$-flags for the tree $T_j$.
\end{proposition}

\begin{proof}
Let $\emptyset = U_0 \subseteq U_1 \subseteq \cdots \subseteq U_{d-1} \subseteq U_d=[N]$ be a flag for the tree $T$. The generating function $g_T(t_1,\ldots,t_d)$ is defined recursively. If the root node lies in $U_1$, we have $t_1$ times the product of the generating functions for each subtree $T_j (1 \leq j \leq k)$. If the root node doesn't lie in $U_1$, then $U_1$ is empty since each $U_i$ is $T$-invariant and all other nodes of the tree eventually fall into the root node. This way, we get the other terms appearing in \eqref{eq gf tree} depending on the least $i$, such that the root node lies in $U_i$.
\end{proof}
The next two results are two of the three main results of this section.
\begin{theorem}
\label{sigma_for_cycle}
Let $T$ be a cycle with $n$ nodes. Then
$$
\sigma(d;T)=g_T(\zeta, \zeta^2, \ldots, \zeta^d).
$$
\end{theorem}

\begin{proof}
Let $W$ be a $d$-splitting subset for the cycle $T$. Fix a node $\tau$ on the cycle. If $\tau$ is in $W$, then the next $(d-1)$ nodes to $\tau$ lie in $TW, \ldots, T^{d-1}W$, respectively. The $d^{th}$ node to $\tau$, say $\tau_d$, may again belong to $W$. Then the next $(d-1)$ nodes to $\tau_d$ lie in $TW,\ldots,T^{d-1}W$ respectively. Continuing this, we obtain a splitting subset of cardinality $m$ since $n=dm$. For the next splitting subset, say $W_1$, we may begin with the node next to $\tau$, say $\tau_1$. Arguing as above, the $(d-1)$ nodes next to $\tau_1$ can not lie in $W_1$ as they belong to $TW_1,\ldots,T^{d-1}W_1$, respectively. Therefore, the second possible node in $W_1$ is $\tau_{d+1}$. Completing the cycle, we obtain another $d$-splitting subset, $W_1$, which is different from $W$ as $\tau_1 \in W_1$ but $\tau_1 \notin W$. Continuing the same argument, we obtain $(d-1)$ distinct $d$-splitting subsets $W_1,\ldots,W_{d-1}$ as $\tau_i \in W_i$ but $\tau_i \notin W_j$ for $i \neq j$. However, as $\tau_d$ belongs to $W$, the $d^{th}$ splitting subset, $W_d$, coincides with $W$ since $T$ is a cycle. Moreover, the cyclic structure of $T$ ensures that these are the only possible splitting subsets. Hence $\sigma(d;T)=d$.
Since $g_T(t_1,t_2,\ldots,t_d)=t_1^n+t_2^n+\cdots+t_d^n$ by Proposition \ref{gf_cycle}, we have that
\begin{align*}
g_T(\zeta,\zeta^2,\ldots,\zeta^{d})=& \zeta^n+\zeta^{2n}+\cdots+\zeta^{dn} =d,
\end{align*}
as $d$ divides $n$.
\end{proof}

\begin{theorem}
Let $T$ be a chain on $[n]$. Then
$$
\sigma(d;T)=g_T(\zeta, \zeta^2, \ldots, \zeta^d).
$$
\end{theorem}

\begin{proof}
Let $W$ be a $d$-splitting subset. Beginning with the top node, say $\tau_1$ of the chain, we attempt to construct a $d$-splitting subset. If $\tau_1 \in W$, then the next $(d-1)$ nodes in the chain, say $\tau_2,\ldots,\tau_d$ belong to $TW,\ldots, T^{d-1}W$ respectively. The next possible node in $W$ is $\tau_{d+1}$. Again, $\tau_{d+2},\ldots,\tau_{2d}$ belong to $TW,\ldots, T^{d-1}W$ respectively. Since $n=dm$, the nodes $\tau_1,\tau_{d+1},\ldots,\tau_{(m-1)d+1}$ constitute a $d$-splitting subset for $T$. Note that the above-defined $d$-splitting subset $W$ is unique because the top node $\tau_1$ does not get fed by any other node, and so it must belong to $W$. This leaves a unique choice of nodes for $W$. Hence $\sigma(d;T)=1$. The result now follows by Corollary \ref{zeta}.
\end{proof}
The following proposition extends Lemma \ref{induc}.
\begin{proposition}
\label{tree 1}
For every tree $T$ which has a $d$-splitting subset, we have
$$
g_T(\zeta, \zeta^2,\ldots,\zeta^{d})=1,
$$
where $g_T(t_1,\ldots,t_d)$ is the generating function for the number of $d$-flags of $T$-invariant subsets.
\end{proposition}

\begin{proof}
We will prove by induction on the number of chains of the tree $T$. If $T$ has one chain and a $d$-splitting subset exists for $T$, then $T$ has $dk$ nodes for some positive integer $k$ and $g_T(t_1,\ldots,t_d)$ is a complete homogeneous symmetric polynomial in $t_1,\ldots,t_d$ of degree $dk$. By Corollary \ref{zeta}, $g_T(\zeta, \zeta^2,\ldots,\zeta^{d})=1$. Suppose the result holds for all trees which consist of less than $k$ chains and have a $d$-splitting subset. Let $T$ be a tree which has $k$ chains such that a $d$-splitting subset exists for $T$. Then consider a chain $C$ beginning from one of the leaves of the tree and ending before the branching node such that $C$ has $dl$ number of nodes for some positive integer $l$, i.e. $T=\widetilde{T} \leftarrow C$, where $\widetilde{T}$ is a tree. Note that such a choice of $C$ is possible because whenever there is a branching that leads to the leaves, there is exactly one branch which has a $d$-splitting subset when joined with the branching node and other branches have $d$-splitting subsets (see Lemma \ref{star}). Then $\widetilde{T}$ is a tree for which a $d$-splitting subset exists and has less than $k$ chains. By the induction hypothesis, 
\begin{equation}
\label{one}
g_{\widetilde{T}}(\zeta, \zeta^2,\ldots,\zeta^{d})=1.
\end{equation}
Also 
\begin{equation}
\label{two}
g_C(\zeta, \zeta^2,\ldots,\zeta^{d})=1.
\end{equation}
Therefore \eqref{one} and \eqref{two} give us $g_T(\zeta, \zeta^2,\ldots,\zeta^{d})=1$.
\end{proof}

We immediately obtain the following corollary using the above result. We will use this corollary to prove Theorem \ref{thm 2}.

\begin{corollary}
\label{cor to star}
Let ${T}$ be a tree with $k$ nodes such that ${T}$ does not have a $d$-splitting subset. Let $\widetilde{T}$ be a rooted tree obtained by joining $T$ to a root node $R$ (i.e., $\widetilde{T}=T \rightarrow R$) such that $\widetilde{T}$ has a $d$-splitting subset, then
$$
\zeta ~g_{{T}}(\zeta, \zeta^2,\ldots, \zeta^d) +\zeta^2 g_{{T}}(\zeta^2,\ldots, \zeta^d) + \cdots + \zeta^{d-1} g_{{T}}(\zeta^{d-1}, \zeta^d) =0.
$$
\end{corollary}

\begin{proof}
The generating function for the tree $\widetilde{T}$ is
$$
g_{\widetilde{T}}(t_1,\ldots,t_d)=t_1 ~g_{{T}}(t_1,\ldots,t_d) + t_2~ g_{{T}}(t_2,\ldots,t_d) + \cdots + t_{d-1} ~g_{{T}}(t_{d-1},t_d) +t_d ~g_{{T}} (t_d).
$$
Since $\widetilde{T}$ is a tree which has a $d$-splitting subset, by Proposition \ref{tree 1}, $g_{\widetilde{T}}(\zeta,\ldots,\zeta^d)=1$. Also, $g_{{T}}(t_d)$ is the generating function for the number of $1$-flags on $[k]$, so $g_{{T}}(t_d)=t_d^k$. Therefore $\zeta^d g_{{T}} (\zeta^d)=(\zeta^d)^{k+1}=1$, and the proposition follows.
\end{proof}
The next result together with Proposition \ref{tree 1} shows that if $T$ is a tree such that $\sigma(d;T)>0$, then $\sigma(d;T)=g_T(\zeta, \zeta^2,\ldots,\zeta^{d})$.
\begin{proposition}
If a $d$-splitting subset exists for a tree, then it is unique.
\end{proposition}

\begin{proof}
Let $T$ be a tree such that $\sigma(d;T)>0$. Let $W$ be a $d$-splitting subset. Since the leaves of the tree are not fed by any other node, they must belong to $W$. Then, the subsequent $(d-1)$ nodes after each leaf can not be in $W$ as they lie in $TW, T^2W,\ldots,T^{d-1}W$, respectively. Cut off the leaves along with the $(d-1)$ nodes following them. This gives us a forest. In the forest, repeat the above procedure with the trees which are not chains with $d$ nodes. After a finite number of steps, we will be left with only chains with $d$ nodes, and $W$ consists of all the top nodes (i.e. the leaves in the final forest). This proves that there is a unique choice of $W$ if $\sigma(d;T)>0$.
\end{proof}

\section{Endofunctions on finite sets}
\label{sec 4}
In this section, we extend the results of the previous section to endofunctions, which have some specific structures. We begin with describing the generating function for the number of $d$-flags of $T$-invariant subsets for a general endofunction $T$.
\begin{lemma}
\label{gf cycle tree}
Let $T$ be an endofunction on $[N]$ with $s$ connected components such that there are $r_i$ nodes in the cycle of $i^{th}$ component. Let $T_{i,1}, T_{i,2},\ldots, T_{i,k_i}$ be the trees attached to the cycle of $i^{th}$ component. Then 
$$
g_T(t_1,\ldots,t_d) = \prod_{i=1}^s \left( \sum_{l=1}^d t_l^{r_i} \prod_{j=1}^{k_i} g_{T_{i,j}}(t_l,t_{l+1},\ldots,t_d) \right),
$$
where $g_{T_{i,j}}(t_l, t_{l+1},\ldots,t_d)$ is the generating function for the number of $(d-l+1)$-flags for the tree $T_{i,j}$.
\end{lemma}

\begin{proof}
Since distinct connected components are independent of each other, we get the product of the generating functions for each connected component. Therefore, it is enough to consider an endofunction with one component.  Let $T$ be an endofunction with a cycle having $r$ nodes, and let $T_1,\ldots,T_k$ be $k$ trees attached to the $r$ nodes of the cycle. Let $\emptyset = U_0 \subseteq U_1 \subseteq \cdots \subseteq U_{d-1} \subseteq U_d=[N]$ be a flag of $T$-invariant subsets. If $U_1$ contains a node from the cycle, then since $U_1$ is $T$-invariant, $U_1$ must have all the cycle nodes. Therefore we get $t_1^{r}$ times the product of the generating functions of the trees attached to the cycle as the nodes of the cycle may serve as the root nodes of the trees. In this case, we get $t_1^r \prod_{j=1}^k g_{T_j}(t_1,\ldots,t_d)$. If $U_1$ doesn't contain a node of the cycle, then since $U_1$ is $T$-invariant and all the trees feed into the nodes of the cycle, $U_1$ must be empty. If $U_2$ contains a node of the cycle, then arguing as above, we obtain $t_2^r \prod_{j=1}^k g_{T_j}(t_2,\ldots,t_d)$. We continue this procedure and obtain all the terms in the sum $\sum_{l=1}^d t_l^r \prod_{j=1}^k g_{T_j}(t_l,\ldots,t_d)$ depending on the first $U_l$, which contains a node from the cycle.
\end{proof}

The following lemma will be useful to prove Theorem \ref{thm 1}.
\begin{lemma}
\label{Riener}
Let $n=dk$ for some positive integers $d$ and $k$. Then for all $1<l \leq d$, we have
$$
h_n(\zeta^l,\zeta^{l+1},\ldots,\zeta^d)=1,
$$
where $h_n$ is the complete homogeneous symmetric polynomial of degree $n$ in $(d-l+1)$ variables.
\end{lemma}

\begin{proof}
Note that
$$
h_n(\zeta^l,\zeta^{l+1},\ldots,\zeta^d) = \zeta^{ln}~ h_n(1,\zeta,\ldots,\zeta^{d-l}).
$$
Let $d-l=m$. Then $m<d$. It is enough to prove that $h_n(1,\zeta,\ldots,\zeta^{m})=1$ for $m<d$, since $\zeta^{ln}=1$. Consider the principal specialization of the homogeneous symmetric polynomial. By \cite[Prop. 7.8.3]{MR1676282},
$$
 h_n(1,q,\ldots,q^{m}) = \gauss{m+n}{n},
$$
where $\gauss{}~{}$ denotes the $q$-binomial coefficient. To prove the result, it suffices to show that $\gauss{m+n}{n}$ at $q=\zeta$ is $1$. By \cite[Prop. 4.2 (iii)]{MR2087303},
$$
{\genfrac{[}{]}{0pt}{}{m+n}{n}_{q=\zeta}} = {{n/d + \lfloor m/d \rfloor}\choose{\lfloor m/d \rfloor}} = {n/d \choose 0} =1.
$$
\end{proof}
With the aid of the above lemma, we obtain the following result, which extends Proposition \ref{tree 1}.
\begin{proposition}
\label{zeta_hn}
Let $T$ be a tree for which a $d$-splitting subset exists and $l$ be a positive integer such that $1<l \leq d$. Let $g_T(t_l,t_{l+1},\ldots,t_d)$ be the generating function for the number of $(d-l+1)$-flags of $T$-invariant subsets. Then
$$
g_T(\zeta^l,\zeta^{l+1},\ldots,\zeta^d)=1.
$$
\end{proposition}

\begin{proof}
We prove by induction on the number of chains of $T$. If $T$ has one chain and a $d$-splitting subset exists for $T$, it has $dk$ number of nodes for some positive integer $k$. In this case, $g_T(t_l,t_{l+1},\ldots, t_d)$ is a complete homogeneous symmetric polynomial of degree $dk$ in the variables $t_l, t_{l+1},\ldots, t_d$. By Lemma \ref{Riener}, $g_T(\zeta^l,\zeta^{l+1},\ldots,\zeta^d)=1$. Suppose the result holds for all trees with less than $k$ chains and a $d$-splitting subset. Let $T$ be a tree for which a $d$-splitting subset exists such that $T$ has $k$ chains. Then we can decompose $T$ as $T=\widetilde{T} \leftarrow C$, where $C$ is a chain beginning from a leaf of the tree and running upto its branching node (branching node excluded) such that $C$ has $ds$ number of nodes for some positive integer $s$, and $\widetilde{T}$ is a tree. Note that such a choice is possible because whenever a branching occurs in $T$, exactly one branch will have a $d$-splitting subset together with the branching node, and other branches will have $d$-splitting subsets (Lemma \ref{star}). $\widetilde{T}$ is a tree with a $d$-splitting subset and less than $k$ chains. By the induction hypothesis, $g_{\widetilde{T}}(\zeta^l,\zeta^{l+1},\ldots,\zeta^d)=1$. Since $g_{C}(\zeta^l,\zeta^{l+1},\ldots,\zeta^d)=1$, the result follows.
\end{proof}

The following two theorems are the main results of this section.
\begin{theorem}
\label{thm 1}
Let $T$ be an endofunction on $[n]$ with a cycle having $ds$ nodes and $k$ trees attached to the nodes of the cycle such that each tree has a $d$-splitting subset. Then
$$
\sigma(d;T)=g_T(\zeta, \zeta^2,\ldots,\zeta^d).
$$
\end{theorem}

\begin{proof}
Since each tree has a $d$-splitting subset and the central cycle has $ds$ nodes, $\sigma(d;T)=d$ (see Theorem \ref{sigma_for_cycle}). We shall prove that $g_T(\zeta, \zeta^2,\ldots, \zeta^d)=d$. Let $T_1, \ldots, T_k$ be $k$ trees attached to the nodes of the central cycle. By Lemma \ref{gf cycle tree},
$$
g_T(t_1,\ldots,t_d) = \sum_{l=1}^d t_l^{ds} \prod_{j=1}^{k} g_{T_j}(t_l,t_{l+1},\ldots,t_d).
$$
Therefore
$$
g_T(\zeta, \zeta^2,\ldots,\zeta^d) = \sum_{l=1}^d 1 =d,
$$
where the last equation follows by Proposition \ref{zeta_hn} and Proposition \ref{tree 1}. 
\end{proof}

\begin{theorem}
\label{thm 2}
Let $T$ be an endofunction on $[n]$ with a cycle having $(ds + 1)$ nodes and let $T_1,\ldots, T_k$ be $k$ trees attached to the nodes of the cycle such that each of the $(k-1)$ trees $T_1,\ldots,T_{k-1}$ has a $d$-splitting subset, and $T_k$ together with its root node on the cycle, has a $d$-splitting subset. Then
$$
\sigma(d;T)=g_T(\zeta, \zeta^2,\ldots,\zeta^d).
$$
\end{theorem}

\begin{proof}
Since $T_k$ has a $d$-splitting subset when joined with a cycle node, this leaves a unique choice of nodes from the cycle to constitute the $d$-splitting subset. Therefore $\sigma(d;T)=1$. We will show that $g_T(\zeta,\zeta^2,\ldots,\zeta^d)=1$. By Lemma \ref{gf cycle tree}, we have
\begin{align*}
g_T(t_1,\ldots,t_d)= \sum_{j=1}^d t_j^{ds+1} \prod_{i=1}^{k} g_{T_i}(t_j, t_{j+1}, \ldots,t_d).
\end{align*}
By Propositions \ref{tree 1} and \ref{zeta_hn}, $g_{T_i}(\zeta^j, \zeta^{j+1}, \ldots,\zeta^d) =1 ~ \forall ~ 1 \leq j \leq d$ and $1 \leq i \leq k-1$. Therefore,
\begin{align*}
g_T(\zeta,\ldots,\zeta^d) =~& \zeta ~g_{{T_k}}(\zeta,\ldots,\zeta^d) + \zeta^2 ~g_{{T_k}}(\zeta^2,\ldots,\zeta^d) \\& + \cdots + \zeta^{d-1}~ g_{{T_k}}(\zeta^{d-1},\ldots,\zeta^d) + \zeta^d ~g_{{T_k}}(\zeta^d).
\end{align*}
By Corollary \ref{cor to star}, $g_T(\zeta,\ldots,\zeta^d) = \zeta^d ~g_{{T_k}}(\zeta^d) = (\zeta^d)^{l+1} =1$,
where $l$ is the number of nodes in the tree ${T_k}$.
\end{proof}

\begin{remark}
We will call the type of endofunctions described in Theorem \ref{thm 1} {\it endofunctions of Type $1$}, while those described in Theorem \ref{thm 2} as {\it endofunctions of Type $2$}.
\end{remark}
The next corollary generalizes Theorems \ref{thm 1} and \ref{thm 2}. It shows that the result holds for a more general class of endofunctions.
\begin{corollary}
\label{cor}
Let $T$ be an endofunction on $[n]$ consisting of $k$ connected components $T_1,\ldots, T_k$ of size $dm_1,\ldots, dm_k$ respectively, such that $m_1+ \cdots + m_k=n/d$ and each $T_i$ is either of Type $1$ or of Type $2$. Then
$$
\sigma(d;T)=g_T(\zeta, \zeta^2,\ldots,\zeta^d).
$$
\end{corollary}

\begin{proof}
Since each $T_i$ is independent of $T_j$ for $i \neq j$, we have 
$$
g_T(t_1,\ldots, t_d) = \prod_{i=1}^k g_{T_i}(t_1,\ldots,t_d).
$$
Also, by Theorem \ref{thm 1} and Theorem \ref{thm 2},
$$
\sigma(d;T_i)=g_{T_i}(\zeta, \zeta^2,\ldots,\zeta^d) ~ \forall ~ 1 \leq i \leq k.
$$
As $m_1+ \cdots + m_k=n/d$, we have that
$$
\sigma(d;T)= \prod_{i=1}^k \sigma(d;T_i),
$$
and the result follows.
\end{proof}

We remark that if the endofunction $T$ is neither of type $1$ nor of type $2$, then
$
\sigma(d;T)
$
may not be equal to $ g_T(\zeta, \zeta^2,\ldots,\zeta^d)$.
The following is an example of an endofunction $T$ with $3k+2$ nodes in the cycle and some trees feeding into the cycle, but
$
\sigma(d;T) \neq g_T(\zeta, \zeta^2,\ldots,\zeta^d).
$
\begin{example}
Consider the endofunction $T:[6] \rightarrow [6]$ defined as $1 \rightarrow 2$, $2 \rightarrow 3$, $3 \rightarrow 4$, $4 \rightarrow 5$, $5 \rightarrow 6$, and $6 \rightarrow 2$  as shown in Fig. \ref{fig 3}.
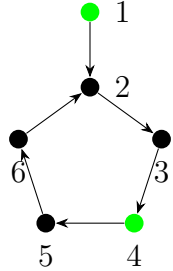
\begin{figure}[h]
\caption{An endofunction on $[6]$.}
\label{fig 3}
  \begin{center}
    \begin{minipage}{\textwidth}
      \centering
      \begin{tikzpicture}
        [scale=1,every node/.style={circle,fill=black},inner sep=2.5pt, minimum size=6pt]
\node[style={circle,fill=green},label=right:$1$] (1) at (0,2) {};
        \node[label=right:$2$] (2) at (0,1) {};
        \node[label=below:$3$] (3) at (0.9510,0.3090) {};
        \node[style={circle,fill=green},label=below:$4$] (4) at (0.5877,-0.8090) {};
        \node[label=below:$5$] (5) at (-0.5877,-0.8090) {};
        \node[label=below:$6$] (6) at (-0.9510,0.3090) {};
\draw[-{Stealth[slant=0]}]
        (1)  to  (2);
        \draw[-{Stealth[slant=0]}]
        (2)  to  (3);
        \draw[-{Stealth[slant=0]}]
        (4) to  (5);
 \draw[-{Stealth[slant=0]}]
        (5) to (6);
 \draw[-{Stealth[slant=0]}]
        (3)  to  (4);
 \draw[-{Stealth[slant=0]}]
        (6)  to  (2);
      \end{tikzpicture}
\end{minipage}
  \end{center}
\end{figure}
Here the central cycle has $3k+2$ nodes $(k=1)$, and a tree is attached to a node of the cycle. Let $d=3$, then $W=\{ 1, 4 \}$ forms a $3$-splitting subset. We may easily verify that $W$ is the only $3$-splitting subset. But $g_T (t_1, t_2, t_3)= {t_1}^{5}(t_1+t_2+t_3) + {t_2}^{5}(t_2+t_3) + {t_3}^{5}(t_3)$ and $g_T(\omega, \omega^2, \omega^3)= 2+ \omega \neq 1 = \sigma(3;T)$, where $\omega$ is the primitive third root of unity.
\end{example}

\section{Acknowledgements}   
The author sincerely thanks Amritanshu Prasad for several helpful discussions. She extends thanks to Samrith Ram for his guidance. Research support from CSIR, India, is gratefully acknowledged. This work was done when the author was a visiting scholar at The Institute of Mathematical Sciences (IMSc), Chennai, India.




\end{document}